\documentclass[10pt,oneside]{article}

\usepackage{tikz}
\usepackage{forest}

\usepackage{amssymb}
\usepackage{amsmath}
\usepackage{amsfonts}
\usepackage{longtable}
\usepackage{tikz-cd}
\usepackage{mathtools}
\usepackage [english]{babel}
\usepackage [autostyle, english = american]{csquotes}
\MakeOuterQuote{"}

\newtheorem{theorem}{Theorem}[section]
\newtheorem{lemma}[theorem]{Lemma}
\newtheorem{proposition}[theorem]{Proposition}
\newtheorem{corollary}[theorem]{Corollary}
\newtheorem{definition}[theorem]{Definition}

\newtheorem{remark}[theorem]{Remark}

\newenvironment{proof}
{\begin{trivlist}  \item \textsc{Proof:}~} {\hfill $\Box$
\end{trivlist}}

\newenvironment{claim}
{\begin{trivlist}  \item \textsc{Claim}~} {\end{trivlist}}
\newenvironment{proof of claim}
{\begin{trivlist}  \item \textsc{Proof of Claim:}~} {\hfill $\Box$ (\textsc{Claim})
\end{trivlist}}

\newenvironment{proof of theorem}
{\begin{trivlist}  \item \textsc{Proof of Corollary \ref{sym}:}~} {\hfill $\Box$
\end{trivlist}}

\newcommand{\closure}[1]{\ensuremath{\mathrm{cl}}(#1)}
\newcommand{\interior}[1]{\ensuremath{\mathrm{int}}(#1)}

\def \Def {\operatorname{Def}}

\def \res_1 {\operatorname{res_1}}

\def \res {\operatorname{res}}

\def \st {\operatorname{st}}

\def \mod {\operatorname{mod}}

\def\Ind#1#2{#1\setbox0=\hbox{$#1x$}\kern\wd0\hbox to 0pt{\hss$#1\mid$\hss}
\lower.9\ht0\hbox to 0pt{\hss$#1\smile$\hss}\kern\wd0}

\def\Notind#1#2{#1\setbox0=\hbox{$#1x$}\kern\wd0\hbox to 0pt{\mathchardef
\nn=12854\hss$#1\nn$\kern1.4\wd0\hss}\hbox to
0pt{\hss$#1\mid$\hss}\lower.9\ht0 \hbox to
0pt{\hss$#1\smile$\hss}\kern\wd0}

\newcommand{\domain}[1]{\ensuremath{\mathrm{domain}}(#1)}
\newcommand{\range}[1]{\ensuremath{\mathrm{rng}}(#1)}

\newcommand{\ma}{\mathfrak{m}}

\newcommand{\bk}{\mathbf{k}}

\title{A note on definable matchings in o-minimal bipartite graphs}
\author{Jana Ma\v{r}\'{i}kov\'{a}}

\begin{document}

\maketitle

\begin{abstract}
We consider bipartite graphs definable in o-minimal structures, in which the edge relation $G$ is a finite union of graphs of certain measure-preserving maps.

We establish a fact on the existence of definable matchings with few short augmenting paths.
Under the additional assumptions that $G\subseteq [0,1]^n$ and 2-regularity,
this yields the existence of definable matchings covering all vertices outside of a set of arbitrarily small positive measure (Lebesgue measure of the standard part).    
As an application we obtain an approximate 2-cancellation result for the semigroup of definable subsets of $[0,1]^n$ modulo an equivalence relation induced by measure-preserving maps.   



\end{abstract}

\begin{section}{Introduction}
This paper is a first step towards understanding definable matchings in definable bipartite graphs in o-minimal structures. Matchings play an important part in many areas of mathematics, such as the theory of equidecompositions, and we believe they will prove higly relevant in the o-minimal setting as well.

Here, a {\em graph\/} consists of a nonempty set of {\em vertices\/} $V$ and a symmetric, antireflexive relation $E\subseteq V^2$ whose elements are called {\em edges\/}.  So graphs have no loops, no multiple edges, and edges are not oriented.  A {\em bipartite graph\/} is a graph whose set of vertices can be partitioned into two disjoint sets $A$ and $B$ so that each edge has one vertex in $A$ and the other vertex in $B$.  
A {\em matching\/} in a bipartite graph $(A\dot\cup B, E)$ is a subset of $E$ which is the graph of a bijection between a subset of $A$ and a subset of $B$.  A matching is {\em perfect\/} if it covers all vertices, i.e. if it is a bijection of $A$ onto $B$.

Throughout, we let $R$ be an o-minimal expansion of an ordered field.  Definable shall mean definable in $R$.
A graph $(V,E)$ is {\em definable\/} if both $V\subseteq R^n$ and $E\subseteq R^{2n}$ are definable.  A {\em definable bipartite graph\/} $(A \dot\cup B,E)$ is a definable graph with a definable bipartition, i.e. both $A$ and $B$ are definable.





What is the situation like for perfect matchings in (infinite) bipartite graphs without any definability assumptions?
By K\"{o}nig's Theorem, every $k$-regular bipartite graph admits a perfect matching. This is a special case of the infinite (two-sided) Hall-Rado-Hall Theorem, according to which a locally finite bipartite graph admits a perfect matching if it satisfies the marriage condition for finite sets (for each $k$, every $k$-element set of vertices has at least $k$ neighbors) in either part.  
However, the definable versions of these two theorems fail, as evidenced by an example by
Laczkovich \cite{L}.  Laczkovich defines a semilinear graph whose edge relation $E$ is a closed subset of the unit square and which consists of finitely many line segments with slopes $\pm 1$ (in fact, $E$ is, when considered as a space with normalized linear measure -- up to a measure-preserving homeomorphism -- just the unit circle). While $E$ contains a perfect matching by K\"onig's Theorem, Laczkovich shows that it does not contain a Borel matching nor a Lebesgue measurable matching.
This is, roughly, due to the fact that, while the normalized linear measure of a matching $M$ in $E$ would be $\frac{1}{2}$, $M$ would also have to be fixed by a certain map which is essentially an irrational rotation of the circle, hence ergodic.
Given that $E$ is in particular definable in an o-minimal structure, this dashes the hope of a definable analogue of K\"onig, or even Hall-Rado-Hall.

One way around this, in the presence of a measure, is to relax the requirement of the matching being perfect to being perfect only outside of a small set.  This has been done in the Borel case by Lyons and Nazarov in \cite{ln}, p.8, Remark 2.6. (for a detailed exposition of the proof see Wang \cite{wang}). Lyons and Nazarov prove the following.  Below, a set of vertices is {\em independent\/}, if no two vertices in that set are neighbors, i.e. they are not incident with the same edge. For a set of vertices $Y$ and edge relation $G$,  \[N_G (Y)=\{ x\colon \exists y\in Y \; (x,y) \in G
\}.\] 
\begin{theorem}[\cite{ln}]\label{ln}
Let $\mathcal{G}=(X,G)$ be a Borel graph on a standard Borel space with a Borel probability measure $\nu$ that is locally finite, $\nu$-preserving, bipartite, and satisfies the follwing expansion condition:
\[\exists c>1 \mbox{ such that for all independent }Y\subseteq X, \; \nu N_G (Y) \geq c \cdot \nu Y. \]  Then $\mathcal{G}$ has a Borel perfect matching $\nu$-a.e..
\end{theorem}
We are interested in an o-minimal counterpart of this theorem.  While the existence of definable matchings in o-minimal graphs is of interest in its own right, another reason is the following.
The condition of being $\nu$-preserving corresponds, in our setting, to the edge relation being a finite union of graphs of isomorphisms (roughly, definable $C^1$-diffeomorphisms with Jacobian determinant equal to $\pm 1$).  
Such graphs and the question of the existence of a perfect matching in them come up when dealing with 
the semigroup of bounded definable sets modulo the equivalence relation induced by isomorphisms, with the operation being given by disjoint union.  These semigroups are in turn closely linked to a long-standing open question about the existence of invariant measures on definable sets in o-minimal structures. 

We obtain Theorem \ref{mainthm} below, an approximate version of Theorem \ref{ln}, when the measure under consideration is Lebesgue measure of the standard part and when we restrict ourselves to 2-regular graphs. The assumption of 2-regularity replaces the expansion condition in Theorem \ref{ln}, which is never satisfied in the bounded definable setting, given that the bipartition of a definable bipartite graph is assumed to be definable.

\begin{theorem}\label{mainthm}
Let $\mathcal{G}=(A\dot\cup B, G)$ be a definable bipartite $\mu$-preserving graph which is 2-regular and such that $A,B\subseteq [0,1]^n$.  Then for every $\epsilon \in \mathbb{R}^{>0}$ there is a definable matching $M\subseteq G$ covering all vertices of $\mathcal{G}$ outside of a set of $\mu$-measure $<\epsilon$. 
\end{theorem}
The proof follows the general outline of the proof in \cite{wang}. In particular, we first prove the existence of matchings with few short augmenting paths, adapting an argument by Elek and Lippner \cite{el}.  We start with the 
archimedean case.  This is Proposition \ref{shortaps} -- the measure under consideration is Lebesgue measure and there is no need to use 2-regularity.  The general case (Theorem \ref{main}) is then derived using results from Ma\v{r}\'ikov\'a \cite{st}, \cite{thesispaper} concerning the structure induced on the residue field by the standard part map.
Theorem \ref{mainthm} it then derived by an argument similar to the one in \cite{wang}, but with 2-regularity yielding an expansion condition.

We remark that Theorem \ref{appl} cannot be improved to yield a definable matching $\mu$-a.e.\ due to the example in \cite{L}.

An application of Theorem \ref{mainthm} concerns cancellation in the semigroup of bounded definable sets modulo the equivalence relation induced by isomorphisms. More precisely, let $B[n]$ be the lattice of bounded definable subsets of $R^n$, and set \[SB[n]=\{ X\in B[n] \colon X\subseteq [-m,m]^n \mbox{ for some } m\in \mathbb{N}\},\] the lattice of strongly bounded definable subsets of $R^n$.
For $X,Y \in SB[n]$ and $\epsilon \in \mathbb{R}^{>0}$, we write $X=_{\epsilon}Y$ iff $\mu (X\triangle Y)<\epsilon$, where $\mu$ is the standard part map composed with Lebesgue measure, and $\triangle$ is symmetric difference.  We write $X=_{a}Y$ iff $X=_{\epsilon}Y$ for all $\epsilon \in \mathbb{R}^{>0}$.  Let $\sim$ be the equivalence relation induced on $B[n]$ by isomorphisms (see Definition \ref{equiv}).
Then $\mathcal{T}_n=B[n]/\sim$ is a semigroup with addition given by disjoint union. For $\alpha, \beta \in \mathcal{T}_n$, we write $\alpha =_a \beta$ iff there are $X\in \alpha$, $Y \in \beta$ such that $X,Y\in SB[n]$ and $X=_{a}Y$.  Then Theorem \ref{mainthm} yields the following.       

\begin{theorem}
Let $\alpha , \beta \in \mathcal{T}_n$. Then $2 \alpha =_a  2 \beta$ implies $\alpha =_a \beta$.

\end{theorem}

\smallskip\noindent
{\bf Some further conventions and definitions.\/}
We let $\mathcal{O}$ be the convex hull of $\mathbb{Q}$ in $R$. Then $\mathcal{O}$ is a valuation ring in $R$ with maximal ideal $\ma$ and residue map $\st \colon \mathcal{O}\to \bk$, where $\bk = \mathcal{O}/\ma$ is the ordered residue field. The residue map extends coordinate-wise to $\st \colon \mathcal{O}^n \to \bk^n $.  If $R$ is sufficiently saturated, then $\mathcal{O}/\ma = \mathbb{R}$ and the residue map is called the standard part map.  In that case, we denote by $\mathbb{R}_{ind}$ the o-minimal structure on $\mathbb{R}$  which is generated by the standard part map, i.e. the ordered field $\mathbb{R}$ expanded by the relations $\st X$, where $X\in \Def^n (R)$ and $\st X := \st (X \cap R^n )$. 

\begin{definition}\label{weakisodef}

\begin{enumerate}
    \item By a {\em measure\/} on $SB[n]$ we mean a finitely additive map $\mu \colon SB[n] \to \mathbb{R}^{\geq 0}$ (addition on $SB[n]$ is given by disjoint union) such that $\mu(\emptyset )=0$.

\item An {\em n-isomorphism\/} is a definable $C^{1}$-diffeomorphism $f\colon U\to f(U)$, where $U\subseteq R^n$ is definable and open, and $|J f(x)|=1$ for all $x\in U$.
    
    \item Given a measure $\mu$ on $SB[n]$, we say that $\mu$ is {\em invariant\/} if $\mu (X)=\mu (f(X))$ whenever $X\in SB[n]$ and $f$ is an $n$-isomorphism.
    
    \item Let $\nu$ be an invariant measure on $SB[n]$. We say that $\mathcal{G}=(V,G)$ is {\em $\nu$-preserving\/} if $\mathcal{G}$ is definable, $V\subseteq \mathcal{O}^n$, and there is a partition of $V$ into cells such that for each open cell $C$ in this partition, $G\cap (C\times R^n)$ is the graph of an $n$-isomorphism.
    \end{enumerate}
\end{definition}
For $R$ sufficiently saturated, we define an invariant measure $\mu$ on $SB[n]$ by assigning to $X\in SB[n]$ the $n$-dimensional Lebesgue measure of its standard part (see \cite{st}, p. 18, proof of Lemma 6.4, for a proof of invariance).

\begin{remark}
It will be easy to see that Theorems \ref{main}, \ref{appl},  and \ref{2cancellation} remain valid if we replace 2.-4. in Definition \ref{weakisodef} by the following, perhaps more natural, notions.

\begin{definition}
\begin{enumerate}
 \item An {\em n-isomorphism\/} is a definable $C^1$-diffeomorphism $f\colon R^n \to R^n$ with $|J f(x)|=1$ for all $x\in R^n$.

 \item Given a measure $\mu$ on $SB[n]$, we say that $\mu$ is {\em invariant\/} if $\mu (X) = \mu (f(X))$ whenever $X\in SB[n]$ and $f$ is an $n$-isomorphism.

\item Let $\nu$ be an invariant measure on $SB[n]$.
We say that a definable graph $\mathcal{G}=(V,G)$ with $V\subseteq \mathcal{O}^n$ is {$\nu$-preserving\/} if there is a partition of $V$ into cells such that for each open cell $C$ in this partition, $G\cap (C\times R^n)$ is the graph of an $n$-isomorphism restricted to $C$.
\end{enumerate}
\end{definition}
\end{remark}
For $x,y \in R^n$ and definable, bounded $X\subseteq R^n$, we let $d(x,y)$ be the euclidean distance between $x$ and $y$, and we set \[d(X,y)=\inf\{d(x,y)\colon x\in X \}.\]
For $x\in R^n$ and $r>0$, we denote by $B_r (x)$ the open ball of radius $r$ centered at $x$, i.e. the set $\{y\in R^n \colon \; d(x,y)<r \}$.
\end{section}

\begin{section}{The archimedean case}
In this section, we assume that the underlying set of $R$ is $\mathbb{R}$.  Then $SB[n]=B[n]$ and Lebesgue measure $\lambda$ is an invariant measure on $B[n]$.
\begin{subsection}{Colorings}
\begin{definition}
Let $\mathcal{G}=(V,G)$ be a definable graph.  We say that a definable map $c\colon V \to X$ is a {\em definable coloring\/} of $\mathcal{G}$ if $X$ is a finite set, and whenever $(v,w)\in 
G$, then $c(v)\not=c(w)$. 
\end{definition}

\begin{lemma}\label{colorings}
Let 
$\mathcal{G}=(V,G)$ with $V\in B[n]$ be a definable graph such that every vertex has finite degree. Then there is a definable coloring of $\mathcal{G}$ outside of a definable subset of the vertex set of arbitrarily small positive $\lambda$-measure, i.e. for every $\epsilon >0$ there is a definable $V'\subseteq V$ with $\lambda (V\setminus V')<\epsilon $ and a definable coloring of $\mathcal{G}' :=(V', G\cap (V')^2 )$. 

\end{lemma}

\begin{proof}
Let $\mathcal{C}$ be a decomposition of $\mathbb{R}^{2n}$ into cells partitioning $G$, and let \[\mathcal{D} = \{ \pi^{2n}_{n}C\colon C\in \mathcal{C} \, \& \, C\subseteq G \, \& \, \dim{C}=n \}.\] Then, because every vertex of $\mathcal{G}$ has finite degree, we may assume that for each $D \in \mathcal{D}$, $G\cap (D \times \mathbb{R}^n )$ is a finite disjoint union of graphs of definable, continuous functions.  Let $\mathcal{F}_{D}$ be the collection of these functions.

\begin{claim}
Let $\epsilon >0$, $D\in \mathcal{D}$ and $f\in \mathcal{F}_D$.  By $\mathcal{G}_{f}$ we denote the graph $(D\cup f(D), \Gamma f )$.  Then there is a definable coloring of $\mathcal{G}_{f}$ outside of a definable subset of $D\cup f(D)$ of $\lambda$-measure $<\epsilon$.  
\end{claim}
\begin{proof of claim}
We set  
\[D_{\delta} := \{ x\in D\colon d(x,\partial{D})\geq \delta \},\] where $\partial D :=\closure D \setminus \interior{D}$,
and $\delta >0$ is such that $\lambda (D\setminus D_{\delta})<\frac{\epsilon}{2}$ and $\lambda (f(D\setminus D_{\delta }))<\frac{\epsilon}{2}$ (the existence of such a $\delta$ follows from the boundedness of the vertex set).  
Define \[F\colon D_{\delta}\to \mathbb{R}^{\geq 0}\colon x\mapsto d(x,f(x)).\]  Then, because 
$\Gamma f|_{D_{\delta}} \subseteq G$,
$G$ is antireflexive, $f$ is continuous and $D_{\delta}$ is closed,  $F$ is bounded away from 0, say by $r>0$.  
Since $D_{\delta}$ is compact, we can find a finite covering $\mathcal{B}$ of $D_{\delta}$ by open balls of radius $\frac{r}{2}$.  For $x \in D_{\delta}$, define $c(x)=i$, where $i$ is the smallest index of a ball from $\mathcal{B}$ containing $x$.  If $x,y \in D_{\delta}$ are such that $c(x)=c(y)$, then $x,y \in B$ for some $B \in \mathcal{B}$, so $d(x,y)<r$ and $x,y$ cannot be neighbors.
\end{proof of claim}
Let $\epsilon >0$.  We shall now define $V'\subseteq V$ with $\lambda (V'\setminus V)<\epsilon$, and find a definable coloring of the graph $(V' , G\cap (V')^2)$.  

Let $|\mathcal{D}|=N$, and let $M$ be an upper bound for the degrees of the vertices of $\mathcal{G}$. 
Since vertices of degree 0 may be colored by any color, we may as well assume that $\mathcal{F}_D \not= \emptyset$ for each $D\in \mathcal{D}$.
For every $D\in \mathcal{D}$ and every $f\in \mathcal{F}_D$, use the claim to find a definable coloring $c_f$ of $\mathcal{G}_f$ outside of a definable $S_f \subseteq D\cup f(D)$ of $\lambda$-measure $<\frac{\epsilon }{MN}$.  We set $V':=V\setminus \bigcup_{D\in \mathcal{D}} \bigcup_{f\in \mathcal{F}_D} S_f$.  Note that \[\lambda (\bigcup_{D\in \mathcal{D}}\bigcup_{f\in \mathcal{F}_D} S_f) < M\cdot N \cdot \frac{\epsilon }{MN}=\epsilon.\]
Define a map $c$ with $\domain{c}=V'$ and range the power set of $\dot\bigcup \range{c_f}$, where the union is taken over all $D\in \mathcal{D}$ and all $f\in \mathcal{F}_D$ as follows.  For $x\in V'$ let $c(x)$ consist of all the $c_f (x)$ with $x\in \domain{c_f}$.  Suppose $x,y \in V'$ and $(x,y) \in G\cap (V')^2$.  Then $f(x)=y$ for some $f\in \mathcal{F}_D$ with $x \in D$. Hence $c_f (x) \not= c_f (y)$, and so $c(x)\not=c(y)$. It follows that $c$ is as required.
\end{proof}
Given a definable graph $\mathcal{G}=(V,G)$ and a measure on $V$, we shall say that $\mathcal{G}$ is {\em definably almost-colorable\/} if for every $\epsilon \in \mathbb{R}^{>0}$, there is $V_{\epsilon }\subseteq V$ with measure of $V\setminus V_{\epsilon}$ less than $\epsilon$,  and a definable coloring of $(V_{\epsilon },G\cap (V_{\epsilon})^2)$.
\begin{remark}
    Lemma \ref{colorings} fails for $R$ non-archimedean:  Let $\epsilon$ be a positive infinitesimal in $R$, and consider the graph $\mathcal{G}=([0,1]_R , G)$, where $(x,y) \in G$ iff $y=x+\epsilon$ and $0\leq x \leq 1-\epsilon$.  Then $\mathcal{G}$ is not definably almost colorable.

\end{remark}

\end{subsection}

\begin{subsection}{Matchings with few short augmenting paths}

Here, we let $\mathcal{G}=(A\dot\cup B,G)$ be a definable, $\lambda$-preserving, bipartite graph with $A,B\in B[n]$, and
$M\subseteq G$ a definable matching.   We fix $K\geq 0$.

We say that a finite set of subsets of a definable $X\subseteq R^n$ is an {\em open partition\/} of $X$, if each of its members is an open cell contained in $X$, and its union covers $X$ outside of a set of $\lambda$-measure 0.
We say that an open partition $\{ X_i \}$ of $X$ partitions a definable $Y\subseteq X$ if a subset of $\{ X_i \}$ constitutes an open partition of $Y$. An open partition $\{Y_j \}$ of $X$ is a refinement of another open partition $\{ X_i \}$ of $X$ if $\{ Y_j \}$ partitions each $X_i$.

\medskip\noindent
{\bf Generating sequences of paths.\/}
We shall define a sequence \[ \mathcal{A}_{0} , \mathcal{A}_2 , \dots , \mathcal{A}_{2K+2}\] of open partitions of $A$, and a sequence \[ \mathcal{B}_{1} , \mathcal{B}_3 , \dots ,\mathcal{B}_{2K+1}\] of open partitions of $B$.  In each sequence, every open partition will be a refinement of its predecessor.

Let $\mathcal{C}$ be a decomposition of $\mathbb{R}^{2n}$ into cells, which partitions both $G$ and $M$.  Set \[ \mathcal{A}_0 = \{ \pi^{2n}_{n}C \colon C\in \mathcal{C} \, \& \, \dim{\pi^{2n}_{n}C}=n  \, \& \, \pi^{2n}_{n}C\subseteq A \}. \]
Since $\mathcal{G}$ is $\lambda$-preserving, we may assume that for each $A_{0,i} \in \mathcal{A}_0$, $G|_{_{A_{0,i}}} := G \cap (A_{0,i} \times \mathbb{R}^n )$ is a finite union of graphs of $n$-isomorphisms.  We denote the set of these isomorphisms by $\mathcal{F}_{0,i}$.

Assuming that for each $A_{2m,i} \in \mathcal{A}_{2m}$, where $0\leq m \leq K$, $\mathcal{A}_{2m}$ and $\mathcal{F}_{2m,i}$ have already been defined, 
we let $\mathcal{B}_{2m+1}$ be an open partition of $B$ partitioning $f(A_{2m,i})$, for each $f\in \mathcal{F}_{2m ,i}$ and each $A_{2m,i} \in \mathcal{A}_{2m}$. For $B_{2m+1,j}\in  \mathcal{B}_{2m+1}$, let $\mathcal{F}_{2m+1,j}$ be the set of all $f^{-1}|_{B_{2m+1,j}}$, with $f\in \bigcup_{i} \mathcal{F}_{2m,i}$, where the sum is taken over all $i$ such that $A_{2m,i} \in \mathcal{A}_{2m}$,  and $B_{2m+1,j} \subseteq \range{f}$.

To define $\mathcal{A}_{2m+2}$, where $0\leq m\leq K$, assume that $\mathcal{B}_{2m+1}$ and $\mathcal{F}_{2m+1}$ have been defined.  Let $\mathcal{A}_{2m+2}$ be an open partition of $A$ which partitions $f^{-1}(B_{2m+1,j})$
for each $f^{-1}\in \mathcal{F}_{2m+1,j}$ and each $B_{2m+1,j} \in \mathcal{B}_{2m+1}$.

\medskip\noindent
Now set 
\begin{equation*}
    \mathcal{C}_i =
    \begin{cases*}
      \mathcal{A}_i & if $i$ is even \\
      \mathcal{B}_i       & if $i$ is odd.
    \end{cases*}
  \end{equation*}
Let $\mathcal{P}$ be the set of paths $p$ in $\mathcal{G}$ of length $l \leq 2K+1$ such that if $p_0 \in A$, then $p_i \in \bigcup \mathcal{C}_i$ for each $i=0,\dots ,l$, and if $p_0 \in B$, then $p_i \in \bigcup \mathcal{C}_{i+1}$ for each $i=0,\dots ,l$.

\begin{definition}
Given a path $p \in \mathcal{P}$ of length $l \leq 2K+1$, the {\em generating sequence\/} of $p$ is the unique sequence $(g_0 , \dots ,g_{l-1})$ of isomorphisms such that
\begin{enumerate}
\item if $p_0 \in \bigcup \mathcal{A}_0$, then each $g_i \in \mathcal{F}_{ i , j}$ for some $j$, and
\item if $p_0 \in \bigcup \mathcal{B}_1$, then each $g_i \in \mathcal{F}_{i+1,j}$ for some $j$, and 
\item $p_{i+1} = g_{i} (p_i ) $ for all $i \in \{0,\dots ,l-1 \}$.
\end{enumerate}
\end{definition}
Let $s$ be the generating sequence of a path $p\in \mathcal{P}$.  We denote by $\mathcal{S}_s$ the set of all possible starting vertices of paths in $\mathcal{P}$ with generating sequence $s$.  Note that $\mathcal{S}_s\subseteq A_{0,j}$ or $\mathcal{S}_s \subseteq B_{1,j}$ for some $j$, and that we may identify $\mathcal{P}$ with $\dot\bigcup_s \mathcal{S}_s$, where the disjoint union is taken over all generating sequences $s$ of paths in $\mathcal{P}$.



We now define a measure $\nu$ on the definable subsets $\mathcal{P}'$ of $\mathcal{P}$.
We have $\mathcal{P'}=\dot\bigcup_s \mathcal{S}'_s$, where $\mathcal{S}'_s \subseteq \mathcal{S}_s$ and $s$ ranges over the generating sequences of paths in $\mathcal{P}$.
We set
\[\nu (\mathcal{P'}):= 
\sum_{\mathclap{\substack{s}}}
 \lambda \mathcal{S}'_s ,\] so $\nu$ on $\mathcal{P}$ is just Lebesgue measure on $\dot\bigcup_s \mathcal{S}_s$.

Let $\mathcal{H}=(\mathcal{P},H)$ be the definable graph with vertex set $\mathcal{P}$ and $(p,q) \in H$ iff $p\not=q$ and $p_k = q_l$ for some $0\leq k,l \leq 2K+1$. 
Since every vertex of $\mathcal{H}$ has finite degree, by Lemma \ref{colorings} we obtain the following.
\begin{lemma}\label{almostcoloring}
The graph $\mathcal{H}$ is definably almost-colorable (with respect to $\nu$).
\end{lemma}

\medskip\noindent
{\bf Augmenting paths.\/} Our aim now is to define a matching $M'$ which covers the vertices covered by $M$, but which has only few "short" augmenting paths.

We shall use the following terminology.  If $p$ is an augmenting path of length $l$ for a matching $\mathcal{M}$ (hence $p_0$ and $p_{l}$ are not covered by $\mathcal{M}$ and each $(p_{2i},p_{2i+1}) \not\in \mathcal{M}$ and each $(p_{2i+1},p_{2i+2}) \in \mathcal{M}$), then by {\em flipping $p$\/} we mean removing the edges $(p_{2i+1},p_{2i+2})$ from $\mathcal{M}$ and instead placing the edges $(p_{2i},p_{2i+1})$ into $\mathcal{M}$. Note that flipping an augmenting path results in a new matching, which covers the same vertices as $\mathcal{M}$ (and two additional ones).
\begin{proposition}\label{shortaps}
Let $\delta \in \mathbb{R}^{> 0}$.  There is a definable matching $M'\subseteq G$ covering the vertices covered by $M$, and not having any augmenting paths of length $\leq 2K+1$ outside of a definable subset of $\mathcal{P}$ of $\nu$-measure $< \delta$.
\end{proposition}

\begin{proof}
Since $\mathcal{G}$ is $\lambda$-preserving, $\dim{A}=\dim{B}$, and we may assume that $\dim{A}=n$.
By Lemma \ref{almostcoloring}, we can find a definable $\mathcal{P}' \subseteq \mathcal{P}$ with $\nu(\mathcal{P} \setminus \mathcal{P}' )<\frac{\delta}{2}$ and a definable coloring $c$ of the graph $(\mathcal{P}' , H\cap \mathcal{P}'^2 )$ with $\range{c} = \{0,1, \dots ,C-1 \}$ for some $C\in \mathbb{N}$.  
Let $a=(a_k )$ be the sequence with $a_k = (k \mod C)$ for each $k\in \mathbb{N}$. 

We shall obtain the desired matching $M'$ as a member of a sequence $M_0 , M_1 , M_2 , \dots$ of definable matchings.  We set $M_0 := M$.  
To obtain $M_{k+1}$ from $M_{k}$, $k\geq 0$, flip all augmenting paths for $M_k$ that are contained in $c^{-1}(a_{k})$.
Given that we are only flipping paths with the same color, in each step we indeed obtain a new matching.  Note that each $M_{k+1}$ covers the vertices covered by $M_k$.
It now suffices to establish the claim below.
\begin{claim} There is $k$ such that $M_k$ has no augmenting paths of length $\leq 2K+1$ outside of a definable subset of $\mathcal{P}$ of $\nu$-measure $<\delta$.  
\end{claim}
\begin{proof of claim}
First note that an edge $(a,b)\in G$ can flip belonging to $M_k$ for only finitely many $k$'s.  This follows by an argument as in \cite{wang}, p.12: Set \[R_a := \{x\colon x \mbox{ is reachable from $a$ in } \leq 2K+1 \mbox{ many steps} \}.\] Then $b\in R_a$, and $R_a$ is finite because every vertex of $\mathcal{G}$ has finite degree.  Every time we flip $(a,b)$, this happens because $(a,b)$ is part of an augmenting path of length $\leq 2K+1$ whose flipping results in an increase of the covered vertices of $R_a$.  But given the finiteness of $R_a$, this can only happen finitely many times.

Also note that since $\mathcal{G}$ is definable, there is a uniform bound, say $N$, on the number of times an edge can flip (take, for instance, $N=d^{2K+1}$, where $d$ is an upper bound on the degrees in $\mathcal{G}$).

It suffices to show that for $i \in \{0,\dots , C-1 \}$ there is $k_i$
such that for every $m\geq k_i$, $c^{-1}(i)$ contains no augmenting paths of length $\leq 2K+1$ for $M_{i+mC}$ outside of a definable subset of $\mathcal{P}$ of $\nu$-measure $<\frac{\delta}{2C}$.

Let $AP_{i+mC} \subseteq \mathcal{P}$ be the set of augmenting paths for $M_{i+mC}$ in $c^{-1}(i)$, and assume to the contrary that there are arbitrarily large $k$ such that $\nu (AP_{i+kC})\geq \epsilon$.  To create $M_{i+kC+1}$, the paths in $AP_{i+mC}$ are flipped, but  because each edge flips at most $N$ times, this can only happen at most $\frac{\nu (c^{-1}(i))}{\epsilon}\cdot N$ times, a contradiction.
\end{proof of claim}
Set $k:=\max \{k_0 , \dots ,k_{C-1}\}$.  Then the set of augmenting paths for $M_k$ of length $\leq 2K+1$ has measure $< \frac{\delta}{2}+C\cdot \frac{\delta}{2C}=\delta$.

\end{proof}

\end{subsection}

\end{section}

\begin{section}{Reduction to the archimedean case}
In this section, we drop the assumption that the underlying set of $R$ is $\mathbb{R}$.  Instead, we assume that $R$ is $(2^{\aleph_0})^{+}$-saturated.  We will use the following definitions and lemma from \cite{st}.

By a {\em $\mathbb{Q}$-ball\/} in $R^n$ we mean an open ball with rational radius, i.e. a ball of the form
\[B_{r }(x) = \{ y\in R^n \colon d(x,y)<r ,\; x\in R^n ,\; r \in \mathbb{Q}^{>0}\}.\] The lemma below is Lemma 4.1, p. 183 in \cite{st} (here, we prefer to state it in terms of $\mathbb{Q}$-balls rather than $\mathbb{Q}$-boxes). 

\begin{lemma}[\cite{st}]\label{Qbox}
Suppose $X\subseteq R^n$ is definable and $\dim( \st{X} ) =n$.  Then $X$ contains a $\mathbb{Q}$-ball. 
\end{lemma}
The following definition (Definition 3.1, p.179, \cite{st}) and theorem 
(a slightly weaker version of Corollary 6.2, p.191, in \cite{st}) will be crucial.
\begin{definition}[\cite{st}]
Given functions $f \colon X \to R$, $X \subseteq  R^n$,
and $F \colon Y \to \mathbb{R}$
with $Y \subseteq \mathbb{R}^n$, we say that $f$ induces $F$ if $f$ is definable (so $X$ is definable),
$Y^h \subseteq X$, $f|C^h$
is continuous, $f(C^h) \subseteq \mathcal{O}$ and $\Gamma F = \st(\Gamma f) \cap (Y \times \mathbb{R})$.

For $f\colon X \to R^n$ and $F\colon Y \to \mathbb{R}^n$, where $X\subseteq R^n$ and $Y \subseteq \mathbb{R}^n$, we say that $f$ induces $F$ if the coordinate functions of $f$ induce the corresponding coordinate functions of $F$.
\end{definition}
\begin{theorem}[\cite{st}]\label{indfunct}
If $f\colon Y\to R$, where $Y\subseteq R^n$ and $\Gamma f \subseteq \mathcal{O}^{n+1}$, is definable, then there is a decomposition $\mathcal{C}$ of $\mathbb{R}^{n}$ into cells that partitions $\st Y$ and such that if $C\in \mathcal{C}$ is open and $C\subseteq \st Y$, then $f$ is continuously differentiable on an open $X\subseteq Y$ containing $\st^{-1}(C)$ and $f,\frac{\partial f}{\partial x_1},\dots ,\frac{\partial f}{\partial x_n}$, as functions on $X$, induce functions $g,g_1 , \dots ,g_n \colon C\to \mathbb{R}$ such that $g$ is $C^1$ and $g_i = \frac{\partial g}{\partial x_i}$ for each $i$. 
\end{theorem} 
Furthermore, we shall use the following results from \cite{thesispaper}, where 1. is Proposition 5.1, and 2. is extracted from the proof of Lemma 2.15.
\begin{proposition}\label{CandI}
 \begin{enumerate}
     \item If $C\in \Def^n (\mathbb{R}_{ind})$ is closed, then there is a $Z\in \Def^n (R)$ such that $\st Z=\closure{C}$.
     \item If $X,Y \in Def^n (R)$ and $X,Y\subseteq \mathcal{O}^n$, then there is $\epsilon >0$ such that $\st X \cap \st Y = \st (X\cap Y^{\epsilon})$, where $Y^{\epsilon } = \{x\in R^n \colon d(x,Y)<\epsilon \}$.
 \end{enumerate}
\end{proposition}
 Below, $\nu$-measure is defined just as in the case when the underlying set of the structure is $\mathbb{R}$, except using $\mu$ rather than $\lambda$. 
\begin{theorem}\label{main}
Let $\mathcal{G}=(A\dot\cup B, G)$ be a definable bipartite, $\mu$-preserving graph, $d$ an upper bound on the degrees of its vertices, and $M\subseteq G$ a definable matching.
Let further $K\in \mathbb{N}$, $ \delta \in \mathbb{R}^{>0}$, and $\epsilon \in \mathbb{R}^{>0}$ subject to $\epsilon <\frac{\delta}{d^{2K+2}}$.  Then there is a definable matching $X\subseteq G$ such that $X$ covers all vertices covered by $M$ outside of a definable set of $\mu$-measure $< \epsilon$ and $X$ has no augmenting paths of length $\leq 2K+1$ outside of a set of $\nu$-measure $<\delta$.
\end{theorem}
\begin{proof}
We may assume that $\dim(\st A)=n=\dim(\st B)$.
Let $\mathcal{D}$ be a decomposition of $R^{2n}$ into cells which partitions $G$ and $M$. Set \[\mathcal{D}_0 = \{\pi^{2n}_{n}D\colon \; D\in \mathcal{D} \, \& \, D\subseteq G \, \& \, \dim{\pi^{2n}_{n}D}=n \}.\] 
We may assume that if $D\in \mathcal{D}$, $D\subseteq G$ is such that $\pi^{2n}_{n}D \in \mathcal{D}_0$, then $D$ is the union of finitely many $\Gamma f$, where each $f$ is an $n$-isomorphism, and we denote the collection of these $n$-isomorphisms on $\pi^{2n}_{n}D$ by $\mathcal{F}_{\pi^{2n}_{n}D}$. 



Set \[G_1 = G \cap \bigcup_{D\in \mathcal{D}_0 } (D\times R^{n}) \mbox{ and } M_1 = M \cap \bigcup_{D\in \mathcal{D}_0 } (D\times R^{n}),\] and let $\mathcal{C}$ be a decomposition of $\mathbb{R}^{2n}$ into cells partitioning $\st G_1$ and $\st M_1$ and such that $\{\pi^{2n}_{n}C\colon C\in \mathcal{C} \}$ partitions each $\st D$ where $D\in \mathcal{D}_0$.  Let 
\[\mathcal{C}_0 = \{ \pi_{n}^{2n} C \colon C\in \mathcal{C} \,\&\, C\subseteq \st G_1 \,\&\, \dim{C}=n \} .\]
Suppose $D\in \mathcal{D}_0$ and $f \in \mathcal{F}_D$, and let $C \in \mathcal{C}_0$ be such that $C\subseteq \st D$.
Then by Theorem \ref{indfunct}, we may assume that $f$ induces an $n$-isomorphism $g \colon C \to g(C)$.
For $C \in \mathcal{C}_0$, 
let $\mathcal{F}_{C}$ be the set of all $g\colon C \to \mathbb{R}^n$ that are induced by some $f\in \mathcal{F}_{D}$, where $D \in \mathcal{D}_0$ and $C \subseteq \st D$.

We set \[G' = \st G_1 \cap \bigcup_{\begin{smallmatrix} C\in \mathcal{C}_0  \end{smallmatrix}} (C\times \mathbb{R}^n).\]  Then $G'$ is the edge relation of the $\mathbb{R}_{ind}$-definable, $\lambda$-preserving, bipartite graph $\mathcal{G}'$ with bipartition $A'$, $B'$, where $A' = \pi^{2n}_{n} G'$ and $B'$ the projection of $G'$ onto the last $n$ coordinates. 
\begin{claim}
The relation $M' = \st M_1 \cap G'$ is a definable matching in $G'$. 
\end{claim}

\begin{proof of claim}
To see that $M'$ is the graph of a function, assume towards a contradiction that $(X,Y), (X,Y') \in M'$ and $Y\not=Y'$.  Then $X \in C$ for some $C\in \mathcal{C}_0$ and $C\subseteq \st{D}$ for some $D\in \mathcal{D}$.
Since $C$ is open, $\st^{-1}(X)\subseteq D$, and there are $x_1 , x_2 \in D$ such that $\st{x_1} = \st{x_2}=X$, and $\st{f(x_1)}=Y$ and $\st{f(x_2 )}=Y'$ for the unique $f\in \mathcal{F}_{D}$ with $\Gamma f \subseteq M_1$, contradicting that $f$ induces a function $C \to \mathbb{R}^n$.

 Suppose now that $(X,Y) , (X',Y) \in M'$ with $X\not=X'$.  
 If there is $C\in \mathcal{C}_0$ such that $X,X' \in C$, then $\st^{-1}X , \st^{-1}X' \subseteq D$ for some $D\in \mathcal{D}$, and for $f \in \mathcal{F}_D$ with $\Gamma f \subseteq M_1$, we have $\st f(x) = \st f(x')$ for some $x\in \st^{-1}X$, $x'\in \st^{-1}X'$, contradicting that $f$ induces an isomorphism $C\to \mathbb{R}^n$.

 So assume $X\in C_1$, $X'\in C_2$, where $C_1 , C_2 \in \mathcal{C}_0$, $C_1 \not= C_2$.  Let $F\colon C_1 \to \mathbb{R}^n$ and $G\colon C_2 \to \mathbb{R}^n$ be induced by $f\in \mathcal{F}_{D_1}$ and by $g\in \mathcal{F}_{D_2}$ respectively, where $\Gamma f, \Gamma g \subseteq M_1$ and $F(X)=Y=G(X')$.  Then there is $\delta >0$ such that $B_{\delta }(X)\subseteq C_1$ and $B_{\delta }(X')\subseteq C_2$ and, since $F$, $G$ are homeomorphisms, $F(B_{\delta }(X))$, $G(B_{\delta }(X'))$ are open subsets of $\mathbb{R}^n$.

 Since $\st^{-1}(B_{\delta} (X))\subseteq D_1$ and $\st^{-1}(B_{\delta}(X'))\subseteq D_2$, we have $B_{\frac{\delta}{2} }(x)\subseteq D_1$ and $B_{\frac{\delta}{2}}(x')\subseteq D_2$, where $x,x'$ are such that $\st{x}=X$ and $\st{x'}=X'$.  So $F(B_{\frac{\delta}{2}} (X))\subseteq \st{f(B_{\frac{\delta}{2}}(x))}$ and $G(B_{\frac{\delta}{2}} (X'))\subseteq \st{g(B_{\frac{\delta}{2}}(x'))}$. Since $Y \in F(B_{\frac{\delta}{2}}(X))\cap G(B_{\frac{\delta}{2} }(X'))$, there is $\epsilon >0$ such that \[ B_{\epsilon }(Y)\subseteq F(B_{\frac{\delta}{2}}(X))\cap G(B_{\frac{\delta }{2}}(X')), \] hence  $B_{\epsilon }(Y) \subseteq \st f(B_{\frac{\delta}{2} }(x)) \cap \st{g(B_{\frac{\delta}{2}}(x'))}$ - a contradiction with 
 \[f(B_{\frac{\delta}{2} }(x)) \cap \, g(B_{\frac{\delta}{2} }(x'))=\emptyset \] and Lemma \ref{Qbox}.

 \end{proof of claim}

\medskip\noindent
By Proposition \ref{shortaps}, we can find an $\mathbb{R}_{ind}$-definable matching $M''\subseteq G'$ such that all augmenting paths outside of a definable $P\subseteq \mathcal{P}$ of $\nu$-measure $<\frac{\delta}{2}$ are of length $> 2K+1$.

Let $\mathcal{C}'$ be a decomposition of $\mathbb{R}^{2n}$ into cells which is a refinement of $\mathcal{C}$ and partitions $M''$, and let $\mathcal{C}'_{0}$ consist of the cells $\pi^{2n}_{n}C$ of dimension $n$ such that $C\in \mathcal{C}'$ and $C\subseteq M''$.  Find $\alpha \in \mathbb{R}^{>0}$ such that \[\Sigma_{C\in \mathcal{C}'_{0}} \lambda (C\setminus C_{\alpha}) <\frac{\delta}{4\cdot d^{2K+1}}<\frac{\epsilon}{2},\] where \[ C_{\alpha} = \{x\in C\colon d(\partial C,x)\geq \alpha \}.\]
By \ref{CandI}, we can find for each $C\in \mathcal{C}'_0$ and $D \in \mathcal{D}_0$ with $C\subseteq \st D$ a definable $D_C \subseteq D$ such that $\st D_C = \closure{C_{\alpha}}$.  Note that \[M''':=M'' \cap \bigcup_{C\in \mathcal{C}'_{0}} (C_{\alpha} \times \mathbb{R}^n )\] covers the same vertices as $M''$ outside of a set of measure $<\epsilon$, hence $M'$ does, too.  Moreover, $M'''$ has no augmenting paths of length $\leq 2K+1$ outside of a subset of $\mathcal{P}$ of $\nu$-measure $<d^{2K+1} \cdot \frac{\delta}{2\cdot d^{2K+1}}+\frac{\delta}{2}=\delta$.

We now define the desired matching $X\subseteq G$ as a subset of \[\bigcup_{C\in \mathcal{C}'_{0}} (D_C \times R^n ) \cap G.\]  For each $C\in \mathcal{C}'_{0}$ and $D$ and $D_C$ as above,  let $f_{D_C}$ be the restriction to $D_C$ of the first function in $\mathcal{F}_{D}$ which induces the function with graph $M'' \cap (C\times \mathbb{R}^n)$.  Then \[X=\bigcup \{\Gamma f_{D_C}\colon C\in \mathcal{C}'_{0} \}.\]
It is left to check that $X$ satisfies the desired properties.
\begin{trivlist}
\item[$X\subseteq G$ is a matching:] The only way $X$ can fail to be a matching is, if there are $x_1 \in D_{C_1}$ and $x_2 \in D_{C_{2}}$ where $C_1 \not= C_2$ and $(x_1 ,y) , (x_2 ,y) \in X$.  But then $(\st x_1 ,\st y) , (\st x_2 ,\st y) \in M''$, so $\st x_1 = \st x_2$, a contradiction with $d(D_{C_1} , D_{C_2})>\ma$.

\item[$X$ has $<\delta$ augmenting paths of length $\leq 2K+1$:]
Let $P$ be a set of augmenting paths for $X$ of length $\leq 2K+1$ and generating sequence $s$. Then we may identify $P$ with the set of starting vertices of the paths in $P$.  Assume that $\mu (P_0 )=\rho >0$.  If suffices to show that then $P$ induces a set of augmenting paths for $M'''$ of length $\leq 2K+1$ and set of starting vertices of $\lambda$-measure $\geq \rho$. Note that we may assume that $P$ is a set of paths in $G_1$. But then it follows straightforwardly from the definitions of $G_1$, $G'$, $X$, $M'''$ and from Lemma \ref{Qbox} that $\mu$-a.e. on $P$, if
$p=(p_0 ,\dots ,p_l ) \in P$, then each $\st p_i$ is a vertex of $G'$; $\st p_0$, $\st p_l$ are not covered by $M'''$; $(\st p_i , \st p_{i+1}) \in G'$; and $(p_i , p_{i+1}) \in X$ iff $(\st p_i , \st p_{i+1}) \in M'''$.

\item[$X$ covers the same vertices covered by $M$ outside of a set of $\mu$-measure $<\epsilon$ :] Let $V\subseteq A\dot\cup B$ be definable and covered by $M$ but not by $X$, and $\mu (V)\geq \epsilon$. Then $V$ is $\mu$-a.e. covered by $M_1$, and hence $\st V$ is covered by $M'$ outside of a set of $\lambda$-measure 0.  Since $\lambda (\st V) >\epsilon $, this yields a contradiction with $\st X = \closure{M'''}$ and $M'''$ covering the same vertices as $M'$ outside of a set of measure $<\epsilon$.

\end{trivlist}

\end{proof}
\end{section}

\begin{section}{Matchings in 2-regular bipartite graphs}
Let $R$ be a sufficiently saturated expansion of a real closed field.
Unless indicated otherwise, we assume that $\mathcal{G}=(A\dot\cup B, G)$ is a definable bipartite graph which is $\mu$-preserving and 2-regular. 
We shall show the existence of a definable matching in $\mathcal{G}$ covering all vertices outside of a set of arbitrarily small positive $\mu$-measure.

\medskip\noindent
We first consider the case in which we are given a definable matching without any short augmenting paths:
Let $K$ be an even integer and let $M\subseteq G$ be a definable matching without augmenting paths of length $\leq 2K+1$.  
Let $Y_0 \subseteq A\dot\cup B$ consist of the vertices not covered by $M$. 
For $k=0,1, \dots ,\frac{K-2}{2}$, we set $Y_{2k+1} := N_{G}(Y_{2k})$ and $Y_{2k+2} := N_M (Y_{2k+1} )$.  
\begin{lemma}\label{counting}
$\mu Y_{K} = K \cdot \mu Y_{0}$.
\end{lemma}
\begin{proof}
We sketch the proof of this lemma; the details can be easily filled in by the reader, using induction and the absence of short augmenting paths.

Let $v\in Y_0$.  We denote by $T_v$ the following tree rooted in $v$.  From now on, we shall assume that $l\in \{ 0,\dots ,\frac{K-2}{2} \}$.  If $x$ is a vertex of $T_v$ at depth $2l$, then $x$ has two children, labeled by the two vertices incident with $x$ in $\mathcal{G}$. If $x$ is at depth $2l+1$, then $x$ has one child, labeled by the vertex incident with $x$ via an edge in $M$ (whose existence follows from $M$ not having any augmenting paths of length $\leq 2K+1$ and $\mathcal{G}$ not having any odd cycles).  For simplicity, we assume that for $l>0$, the left child of $x$ at depth $2l$ is matched and the right child is unmatched. Furthermore, if $x$ is a vertex of $T_v$, then we denote by $T_x$ the maximal subtree of $T_v$ rooted in $x$. 
We denote by $X_l$ the set of labels of vertices in $T_v$ at depth $l$, and by depth we shall always mean depth with respect to $T_v$ (even when talking about a subtree).  We write $d(x)$ for the depth of a vertex $x$.
Note that we have $|X_{2l+1}| = |X_{2l+2}|$.

Let $x_1$ be the left and $y_1$ the right child of $v$ -- see the picture below. 

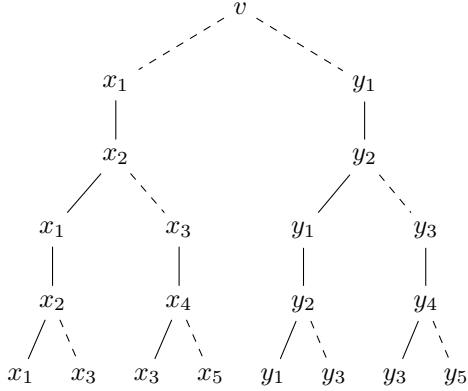
\begin{forest}
[$v$
[$x_1$, edge=dashed [$x_2$ [$x_1$ [$x_2$
[$x_1$][$x_3$, edge=dashed]
]] [$x_3$, edge=dashed [$x_4$
[$x_3$] [$x_{5}$, edge=dashed]
]]]  ]
[ $y_1$, edge=dashed [$y_2$ [$y_1$ [$y_2$ 
[$y_1$] [$y_3$, edge=dashed]
]][$y_3$, edge=dashed [$y_4$
[$y_3$] [$y_5$, edge=dashed]
]]]]
]
\end{forest}

\bigskip

\smallskip\noindent
{\bf Claim 1.\/}
The set of labels of $T_{x_1}$ and the set of labels of $T_{y_1}$ are disjoint.

\noindent
Proof. Suppose $x_i = y_j$, where $d (x_i )\leq d (y_j )$, and there is no label in $T_{x_1}$ appearing in $T_{y_1}$ and being the label of a vertex at depth $\leq d( x_i )$.  
Because $\mathcal{G}$ has no odd cycles, both $d(x_i )$ and $d(y_j )$ are even or both are odd.  We may assume that both are odd.
Now $x_i , y_j$ are each either a matched or an unmatched child of its respective parent, but any possible combination leads to a contradiction with the minimality of $d(x_i )$.  

\smallskip\noindent
{\bf Claim 2.\/}
Suppose $l>0$. Then $|X_{2l}| = 2l$.

\noindent
Proof. Observe that in $T_{x_1}$ and in $T_{y_1}$ respectively, after identifying vertices with same labels at each depth (so a label can only repeat at different depths), at depth $2l+1$, there are exactly two vertices with indegree 1 (so all other vertices have indegree 2).

\smallskip\noindent
The Lemma now follows from the next claim.

\smallskip\noindent
{\bf Claim 4.\/}
Let $v,v'\in Y_0$, $v\not= v'$.  Then the sets of labels at depth $l$ of the tree $T_v$ and of the tree $T_{v'}$ are disjoint.

\end{proof}

\begin{theorem}\label{appl}
Let $\epsilon \in \mathbb{R}^{>0}$.  Then there is a definable matching $M\subseteq G$ such that $M$ covers all vertices outside of a set of measure $< \epsilon$.  
\end{theorem}
\begin{proof} 
Let $\epsilon \in \mathbb{R}^{>0}$ and let $K$ be even, such that $\frac{1}{K} <\frac{\epsilon}{2}$.  By Theorem \ref{main}, we can find a definable matching $M\subseteq G$ with $<\frac{\epsilon}{2}$ augmenting paths of length $\leq 2K+1$. Let $Y_0$ be the set of vertices not covered by $M$, and let $Z\subseteq Y_0$ be the set of starting vertices of augmenting paths of length $\leq 2K+1$.  Then, by the proof of Lemma \ref{counting}, \[K \cdot \mu (Y_0 
\setminus Z) = \mu (Y_0 \setminus Z)_K ,\] where $(Y_0 \setminus Z)_K$ is defined just as $Y_K$, except that one starts with $Y_0 \setminus Z$ instead of $Y_0$ in the inductive definition.
Since $\mu (Y_0 \setminus Z)_{K}$ is bounded above by 1, $\mu (Y_0 \setminus Z) \leq \frac{1}{K}$. So \[\mu Y_0 < \frac{1}{K} + \frac{\epsilon}{2} < \epsilon.\]

\end{proof}
For the application in the next section, we need a version of Theorem \ref{appl} for graphs in which the condition of 2-regularity is relaxed to being 2-regular outside of a set of positive measure (in the application, the positive measure can be assumed to be arbitrarily small).

\begin{corollary}\label{almost2regular}
Let $\mathcal{G} = (A\dot\cup B,G)$ be a definable bipartite graph which is $\mu$-preserving, and such that the degrees its vertices are 2 outside of a set of $\mu$-measure $<\delta \in \mathbb{R}^{>0}$, on which the degrees are equal to 1.  Then for every $\epsilon \in \mathbb{R}^{>\delta }$, $\mathcal{G}$ admits a definable matching covering all vertices outside of a set of measure $<\epsilon$. 
\end{corollary}
\begin{proof}
Let $\epsilon \in \mathbb{R}^{>\delta}$, and let $K \in \mathbb{N}^{>0}$ be such that $\frac{1}{K} < \frac{\epsilon - \delta}{2}$.   Let $M$ be a definable matching in $\mathcal{G}$ with fewer than $\frac{\epsilon - \delta }{2}$ augmenting paths of length $\leq 2K+1$.
The proof of Theorem \ref{appl} can be adapted as follows.  Let $Y_0$ and $T_v$ for $v\in Y_0$ be as in Lemma \ref{counting}.
\begin{claim}
For distinct $v,w\in Y_0$, $T_v$ and $T_w$ have disjoint sets of vertices of degree 1.
\end{claim}
\begin{proof}
In the construction of $T_v$ and $T_w$, vertices appear for the first time along a possible initial segment of an augmenting path starting in $v$.  Suppose that $x_i = y_j$, $\deg{x_i }=1$, where $x_i$ is a vertex of $T_v$, appearing for the first time at depth $i$ and $y_j$ is a vertex of $T_w$, appearing for the first time at depth $j$. We may assume $i\leq j$.  
If $i=0$, then $T_w$ contains an augmenting path of length $\leq K$.
So suppose $i>0$.  Then both $x_i$, $y_j$ are matched or both unmatched, but then $x_{i-1}=y_{j-1}$, a contradiction with the minimality of $i$.  
\end{proof}
Because the same vertex of degree one cannot appear in more than one $T_v$, there is $Y\subseteq Y_0$ of measure $<\delta$ consisting of the roots of trees with a vertex of degree 1.  
Let $Z\subseteq Y_0$ be the set of starting vertices of augmenting paths of length $\leq 2K+1$. Then, by the proof of Theorem \ref{appl}, \[ K\cdot (\mu Y_0 - \delta - \frac{\epsilon - \delta}{2}) = \mu \big( (Y_0 \setminus (Y \cup Z))_K \big).\] So \[ \mu Y_0 < \frac{1}{K} + \delta + \frac{\epsilon - \delta }{2} < \epsilon .\]

\end{proof}

\end{section}

\begin{section}{A cancellation result}
Let $R$ be a sufficiently saturated expansion of a real closed field.
We define an equivalence relation $\sim$ on $B[n]$ as follows.  
\begin{definition}\label{equiv}
Let $X,Y \in B[n]$.  
Then $X\sim Y$ iff there are definable open partitions of $\{ X \}_{1=1}^{k}$, $\{Y_i \}_{i=1}^{k}$ of $X$ and $Y$ respectively, and there are $n$-isomorphisms $f_1 , \dots ,f_k$ such that $Y_i = f(X_i)$ for each $i$. 

We let $\mathcal{B}$ be the semigroup $(B[n]/\sim , + )$, where the binary operation $+$ is given by $a+b = c$, with $c$ the equivalence class containing a disjoint union of $a$ and $b$.    
\end{definition}
The proof of the next theorem is based on the proof of cancellation from Tomkowicz, Wagon \cite{wagon}, p. 177.  While \cite{wagon} uses the Hall-Rado-Hall Infinite Marriage Theorem, we only have Theorem \ref{appl} at our disposition.

\begin{theorem}\label{2cancellation}
Let $\alpha , \beta \in \mathcal{B}$ have strongly bounded representatives and suppose $\alpha +\alpha =_a \beta +\beta $ in $\mathcal{B}$.  Then $\alpha =_{a}\beta $.
\end{theorem}
\begin{proof}
Let $\epsilon \in \mathbb{R}^{>0}$, and
let $A,A'$, and $B,B'$ respectively, be two pairs of disjoint copies of strongly bounded representatives of $\alpha $, and of $\beta $ respectively.  Let $\phi \colon A\to A'$, $\psi \colon B\to B'$ and $\theta \colon A\dot\cup A' \to B\dot\cup B'$ witness $A\sim A'$, $B\sim B'$ and $A\dot\cup A' \sim_{\frac{\epsilon }{2}} B\dot\cup B'$ respectively. 

We define a bipartite graph $\mathcal{H}$ as follows.  The bipartition consists of the two sets \[ \bar{A} = \{ (a,\phi (a))\colon a\in A )\} \mbox{ and } \bar{B} = \{(b,\psi (b))\colon b\in B \}, \] and we let $(a,\phi (a))$ be incident with $(b,\psi (b))$ iff $\theta (a)=b$ or $\theta (\phi (a)) =b$ or $\theta (a)=\psi (b)$ or $\theta (\phi (a)) =\psi (b)$.  
Then $\mathcal{H}=(\bar{A}\cup \bar{B},H)$ is definable bipartite and $\mu$-preserving, and every vertex is of degree $\leq 2.$
To construct a map witnessing $\alpha =_{\epsilon } \beta $, it will suffice to find a definable matching $M\subseteq H$ covering $\bar{A}\cup \bar{B}$ outside of a set of $\mu$-measure $<\epsilon$. 

Let $\bar{A}_0$ and $\bar{B}_0$ be the subsets of $\bar{A}$ and of $\bar{B}$ respectively of vertices of degree 0.  Note that $\mu (\bar{A}_0 \cup \bar{B}_0)<\frac{\epsilon}{2}$, and replace $\bar{A}$ with $\bar{A}\setminus \bar{A}_0$ and $\bar{B}$ with $\bar{B} \setminus \bar{B}_0$ in $\mathcal{H}$.

Let $\bar{A}_1 \subseteq \bar{A}$ and $\bar{B}_1 \subseteq \bar{B}$ be the sets of vertices of degree 1 that are incident with another vertex of degree 1.  Then $H\cap (\bar{A}_1 \times \bar{B}_1 )$ is the graph of a bijection $\bar{A}_1 \to \bar{B}_1$, so there is no harm in replacing $\mathcal{H}$ with the graph  \[\big( (\bar{A} \setminus \bar{A}_1 ) \dot\cup (\bar{B}\setminus \bar{B}_1) \big), H \cap \big( (\bar{A} \setminus \bar{A}_1 ) \times (\bar{B}\setminus \bar{B}_1 ) \big).\] 
Then the remaining vertices of degree 1 in $\bar{A}$ are $(a,\phi (a))$ such that $\theta (a) \uparrow$ and $\theta (\phi (a)) \downarrow$, or vice versa, and similarly for the vertices of degree 1 in $\bar{B}$.
So $\mathcal{H}$ is 2-regular outside of a set of measure $<\frac{\epsilon}{2}$.  

By Corollary \ref{almost2regular}, $H$ contains a definable matching $M$ covering all vertices outside of a set of $\mu$-measure $<\frac{\epsilon}{2}$, hence covering all vertices of the original $\mathcal{H}$ outside of a set of $\mu$-measure $<\epsilon$. 

\end{proof}





\end{section}

\end{document}